\documentclass[a4paper,reqno,11pt]{amsart}
\usepackage{amsfonts,amsmath,amsthm,amssymb,stmaryrd}
\usepackage{hyperref}
\usepackage{color}

\newtheorem{Theorem}{Theorem}[section]
\newtheorem{Lemma}[Theorem]{Lemma}
\newtheorem{Proposition}{Proposition}[section]

\numberwithin{equation}{section} \allowdisplaybreaks
\allowdisplaybreaks 
\begin{document}
\author{Yi Peng}
\address{College of Mathematics and Statistics, Chongqing University,
                             Chongqing, 401331,  China.}
\email[Y. Peng]{20170602018t@cqu.edu.cn}
\author{Huaqiao Wang}
\address{College of Mathematics and Statistics, Chongqing University,
                             Chongqing, 401331,  China.}
\email[H.Q. Wang]{wanghuaqiao@cqu.edu.cn}

\title[Derivation of Navier-Stokes equations]
{Rigorous derivation of the compressible Navier-Stokes equations from the two-fluid Navier-Stokes-Maxwell equations}
\thanks{Corresponding author: wanghuaqiao@cqu.edu.cn}
\

\begin{abstract}
In this paper, we rigorously derive the compressible one-fluid Navier-Stokes equation from the scaled compressible two-fluid Navier-Stokes-Maxwell equations locally in time under the assumption that the initial data are well prepared. We justify the singular limit by proving the uniform decay of the error system, which is obtained by elaborate energy estimates.
\end{abstract}

\maketitle
\section{Introduction}
Besse-Degond-Deluzet \cite{bdd} derived the scaled two-fluid Euler-Maxwell equation. Based on this work, assuming that the electrons and ions density are equal to $n$, which implies ${\rm div}E=0$, and that the electron-neutral and ion-neutral collision frequencies $\nu_{i}$, $\nu_{e}$ are ignored, we have
\begin{subequations}\label{twofiuld nsm}
\begin{align}
&\partial _tn+{\rm div}(nu_{i})=0,\label{1.1a}\\
&\tau\varepsilon\left(\partial _t\left(nu_{e}\right)+{\rm div}\left(nu_{e}\otimes u_{e}\right)-\mu\Delta u_{e}-(\mu+\lambda)\nabla {\rm div}u_{e}\right)+\eta\nabla P_{e}(n)\notag\\
&\quad\!\! =-\kappa^{-1}n\left(E+u_{e}\times B\right)-\frac{\kappa_{ei}\beta}{\kappa^{2}}Kn^{2}(u_{e}-u_{i}), \label{1.1b}\\
&\tau\left(\partial _t\left(nu_{i}\right)+{\rm div}\left(nu_{i}\otimes u_{i}\right)-\mu\Delta u_{i}-(\mu+\lambda)\nabla {\rm div}u_{i}\right)+\eta\nabla P_{i}(n)\notag\\
&\quad\!\! =\kappa^{-1}n\left(E+u_{i}\times B\right)-\frac{\kappa_{ei}\beta}{\kappa^{2}}Kn^{2}(u_{i}-u_{e}), \label{1.1c}\\
&\alpha\partial _tE-\nabla \times B=-\beta j,\;\;{\rm div}E=0,\label{1.1d}\\
&\partial _tB+\nabla \times E=0,\;\;{\rm div}B=0,\label{1.1e}\\
&\kappa j=n(u_i-u_e).\label{1.1f}
\end{align}
\end{subequations}
Where $u_{i}$, $u_{e}$ denote the velocities, $P_{i}$, $P_{e}$ stand for the pressures, $E$, $B$ are the electric and magnetic field, $j$ is the current density. The constants $\mu$ and $\lambda$ are the viscosity coefficients of the flow satisfying $\mu>0$ and $2\mu+3\lambda>0$. $\varepsilon$ is the mass ratio of electron to ion, $\tau$ is the mean time between ion-neutral collisions, $\eta$ is the measure of the thermal energy, $\kappa$ denotes the number of electron-neutral (or ion-neutral) collisions relative to a rotation period of an electron in the unit of $B$ field, $\kappa_{ei}$ stands for measure of the strength of electron-ion collisions. $\beta$ measures the relative strength of the induced $B$ field to the unit $B$ field, $\alpha$ denotes the squared reciprocal of light speed. The rate constant $K$ evaluate collisions between electrons and ions. Equation \eqref{1.1a} is the mass balance law and \eqref{1.1b}--\eqref{1.1c} are momentum balance laws, while \eqref{1.1d}--\eqref{1.1e} are Maxwell equations.

Noted that using the fact that ${\rm div}j=0$ deduced by \eqref{1.1d} and \eqref{1.1f} we infer that
\begin{align}\label{1.1a1}
\partial _tn+{\rm div}(nu_{e})=0.
\end{align}
Setting
$$ \frac{1}{\varepsilon}P_{e}=P_{i}=P,\; u=u_{i}+\varepsilon u_{e},\;\tilde{j}=\frac{j}{n},$$
\eqref{twofiuld nsm} can be recast as
\begin{equation}\label{twofiuld nsm4}
\begin{cases}
\partial _tn+\frac{1}{1+\varepsilon}{\rm div}(nu)=0,\\
\partial _t(nu)+\frac{1}{1+\varepsilon}\left({\rm div}\left(nu\otimes u\right)+\varepsilon\kappa^{2}{\rm div}\left(n\tilde{j}\otimes \tilde{j}\right)\right)\\
\quad\!\!-\mu\Delta u-(\mu+\lambda)\nabla {\rm div}u+\frac{(1+\varepsilon)\eta}{\tau}\nabla P(n)
=\frac{1}{\tau}n\tilde{j}\times B, \\
\kappa\partial _t(n\tilde{j})+\frac{\varepsilon-1}{1+\varepsilon}\kappa^{2}{\rm div}\left(n\tilde{j}\otimes \tilde{j}\right)\\
\quad\!\!+\frac{\kappa}{1+\varepsilon}\left({\rm div}\left(nu\otimes \tilde{j}\right)+{\rm div}\left(n\tilde{j}\otimes u\right)\right)
-\mu\kappa\Delta\tilde{j}-(\mu+\lambda)\kappa\nabla {\rm div}\tilde{j}\\
\quad\!\!=\frac{1+\varepsilon}{\tau\varepsilon\kappa}nE+\frac{1}{\tau\varepsilon\kappa} nu\times B+\frac{\varepsilon-1}{\tau\varepsilon} n\tilde{j}\times B
-\frac{1+\varepsilon}{\tau\varepsilon\kappa}\kappa_{ei}K\beta n^{2}\tilde{j}, \\
\alpha\partial _tE-\nabla \times B=-\beta n \tilde{j},\;\;{\rm div}E=0,\\
\partial _tB+\nabla \times E=0,\;\;{\rm div}B=0,
\end{cases}
\end{equation}
where the second equation is obtained by the summation of $\frac{1}{\tau}$\eqref{1.1b} and $\frac{1}{\tau}$\eqref{1.1c}, while the third one is obtained by the summation of $-\frac{1}{\tau\varepsilon}$\eqref{1.1b} and $\frac{1}{\tau}$\eqref{1.1c}.
Inserting assumptions $\beta=\alpha^{2},\;\alpha=\kappa^{2}$ and scalings $B\rightarrow \kappa^{2}B,\;E\rightarrow\kappa E$ into \eqref{twofiuld nsm4}, we obtain
\begin{equation}\label{twofiuld nsm1}
\begin{cases}
\partial _tn+\frac{1}{1+\varepsilon}{\rm div}(nu)=0,\\
\partial _t(nu)+\frac{1}{1+\varepsilon}\left({\rm div}\left(nu\otimes u\right)+\varepsilon\kappa^{2}{\rm div}\left(n\tilde{j}\otimes \tilde{j}\right)\right)\\
\quad\!\!-\mu\Delta u-(\mu+\lambda)\nabla {\rm div}u+\frac{(1+\varepsilon)\eta}{\tau}\nabla P(n)
=\frac{\kappa^{2}}{\tau}n\tilde{j}\times B, \\
\kappa\partial _t(n\tilde{j})+\frac{\varepsilon-1}{1+\varepsilon}\kappa^{2}{\rm div}\left(n\tilde{j}\otimes \tilde{j}\right)\\
\quad\!\!+\frac{\kappa}{1+\varepsilon}\left({\rm div}\left(nu\otimes \tilde{j}\right)+{\rm div}\left(n\tilde{j}\otimes u\right)\right)
-\mu\kappa\Delta\tilde{j}-(\mu+\lambda)\kappa\nabla {\rm div}\tilde{j}\\
\quad\!\!=\frac{1+\varepsilon}{\tau\varepsilon}nE+\frac{1}{\tau\varepsilon}\kappa nu\times B+\frac{\varepsilon-1}{\tau\varepsilon}\kappa^{2} n\tilde{j}\times B
-\frac{1+\varepsilon}{\tau\varepsilon}\kappa_{ei}K\kappa^{3}n^{2}\tilde{j}, \\
\kappa\partial _tE-\nabla \times B=-\kappa^{2} n \tilde{j},\;\;{\rm div}E=0,\\
\kappa\partial _tB+\nabla \times E=0,\;\;{\rm div}B=0.
\end{cases}
\end{equation}

Extensive research has been done on two-fluid flows. The local existence of classical solution of \eqref{twofiuld nsm} can be deduced by \cite{k,vh,z}. For more results about the well-poseness to the symmetrizable hyperbolic equations, see \cite{k1,m}. In the case of $n_i\neq n_e$ and the viscosity coefficients $\mu=\lambda=0$, \eqref{twofiuld nsm} be the two-fluid compressible Euler-Maxwell equation. There have been a lot of studies about the asymptotic analysis. For instance, Peng-Wang \cite{pw1} formally established the zero-relaxation limit, the non-relativistic limit and the combined non-relativistic and quasi-neutral limit of the Euler-Maxwell system. Later, all of these results have been justified rigorously. More precisely, Peng-Wang-Gu \cite{pwg} focused on the one-fluid Euler-Maxwell system and derived the drift-diffusion equations under time variable scaling by asymptotic expansions. A same result can be extended to two-fluid Euler-Maxwell system since the energy estimates for Euler equations and Maxwell equations can be carried out separately. For the non-relativistic limit, Yang-Wang \cite{yw} obtained the two-fluid compressible Euler-Poisson system while for the combined non-relativistic and quasi-neutral limit, Li-Peng-Xi \cite{lpx} obtained the one-fluid compressible Euler equations. For more results on asymptotic limits with small parameters of two-fluid flows, we refer to \cite{jllj,jjll,lpw,yw2} and the references therein. For the asymptotic limit problem with boundary effects, see \cite{jl,l} and the references therein.

For the one-fluid flows, Peng-Wang \cite{pw2} considered the combined non-relativistic and quasi-neutral limit of the Euler-Maxwell system in the scaling case if the small parameters $\gamma$ and $\varepsilon$ are equal (where $\gamma$ is the inversely proportional to the light speed and $\varepsilon$ stands for the scaled Debye length), and they obtained the incompressible Euler equation by asymptotic expansions. As an improved work, Peng-Wang \cite{pw} rigorously proved the e-MHD equations as the quasi-neutral limit of compressible one-fluid Euler-Maxwell equations through an elaborate nonlinear energy method. Noted that this result can be extended to the combined non-relativistic and quasi-neutral limit without any relation between $\gamma$ and $\varepsilon$. Later, Jiang-Li \cite{jf} verified the compressible MHD limit of the electromagnetic fluid system. Yang-Wang \cite{yw1} obtained the incompressible Navier-Stokes equations as the combined non-relativistic and quasi-neutral limit of compressible Navier-Stokes-Maxwell equations. For more results on asymptotic limits with small parameters for one-fluid flows, see \cite{cjw,pw3}. For the asymptotic limit problem in a bounded domain, see \cite{ghr,ghr1,pw4,v,x} and the references therein.

For the asymptotic convergence of \eqref{twofiuld nsm} mentioned in \cite{bdd}, there are few theoretical justifications. In this paper, we derive the compressible Navier-Stokes equation from \eqref{twofiuld nsm} rigorously as $\kappa\rightarrow0$, which means that the density of the plasma is large, and $\varepsilon\rightarrow0$.

Before proceeding, let us first introduce the notations and lemmas used throughout this paper. We denote by $H^{l}(\mathbb{R}^{3})$ the standard Sobolev's space in the whole space $\mathbb{R}^{3}$, and shorted by $\left\|\cdot\right\|_{l}$ the norm of the Banach space $H^{l}(\mathbb{R}^{3})$. Let $\alpha=(\alpha_{1},\alpha_{2},\alpha_{3})$, we denote $\partial^{\alpha}=\partial^{\alpha_{1}}\partial^{\alpha_{2}}\partial^{\alpha_{3}}$ and $\left|\alpha\right|=\alpha_{1}+\alpha_{2}+\alpha_{3}$. $A\lesssim B$ means that $A\le CB$ for a constant $C>0$.

The following basic Moser-type calculus inequalities \cite{km,km1} will be used frequently in the proof of the main result.
\begin{Lemma}\label{lem1.1}
For any nonnegative multi-index $\alpha$ with $|\alpha|\leq s$ and $f,g\in H^{s}$, it holds
$$
\left\|\partial_{x}^{\alpha}\left(fg\right)\right\|_{L^{2}}\leq C_{s}\left(\left\| f\right\|_{L^{\infty}}\left\|D_{x}^{s}g\right\|_{L^{2}}+\left\|g\right\|_{L^{\infty}}\left\|D_{x}^{s}f\right\|_{L^{2}}\right),$$
and
$$\left\|\partial_{x}^{\alpha}\left(fg\right)-f\partial_{x}^{\alpha}g\right\|_{L^{2}}\leq C_{s}\left(\left\|D_{x}^{1}f\right\|_{L^{\infty}}\left\|D_{x}^{s-1}g\right\|_{L^{2}}+\left\|g\right\|_{L^{\infty}}\left\|D_{x}^{s}f\right\|_{L^{2}}\right).$$
\end{Lemma}

Next, we will provide a formal asymptotic analysis of \eqref{twofiuld nsm1}, which will shed light on its links with the compressible Navier-Stokes equation. In order to emphasize the unknowns depending on the singular perturbation parameter $\kappa$, we rewrite the system \eqref{twofiuld nsm1} as
\begin{equation}\label{twofiuld nsm2}
\begin{cases}
\partial _tn^{\kappa}+\frac{1}{1+\varepsilon}{\rm div}(n^{\kappa}u^{\kappa})=0,\\
\partial _t(n^{\kappa}u^{\kappa})+\frac{1}{1+\varepsilon}\left({\rm div}\left(n^{\kappa}u^{\kappa}\otimes u^{\kappa}\right)+\varepsilon\kappa^{2}{\rm div}\left(n^{\kappa}\tilde{j}^{\kappa}\otimes \tilde{j}^{\kappa}\right)\right)\\
\quad\!\!-\mu\Delta u^{\kappa}-(\mu+\lambda)\nabla {\rm div}u^{\kappa}+\frac{(1+\varepsilon)\eta}{\tau}\nabla P(n^{\kappa})
=\frac{\kappa^{2}}{\tau}n^{\kappa}\tilde{j}^{\kappa}\times B^{\kappa}, \\
\kappa\partial _t(n^{\kappa}\tilde{j}^{\kappa})+\frac{\varepsilon-1}{1+\varepsilon}\kappa^{2}{\rm div}\left(n^{\kappa}\tilde{j}^{\kappa}\otimes \tilde{j}^{\kappa}\right)\\
\quad\!\!+\frac{\kappa}{1+\varepsilon}\left({\rm div}\left(n^{\kappa}u^{\kappa}\otimes \tilde{j}^{\kappa}\right)+{\rm div}\left(n^{\kappa}\tilde{j}^{\kappa}\otimes u^{\kappa}\right)\right)
-\mu\kappa\Delta\tilde{j}^{\kappa}-(\mu+\lambda)\kappa\nabla {\rm div}\tilde{j}^{\kappa}\\
\quad\!\!=\frac{1+\varepsilon}{\tau\varepsilon}n^{\kappa}E^{\kappa}+\frac{1}{\tau\varepsilon}\kappa n^{\kappa}u^{\kappa}\times B^{\kappa}+\frac{\varepsilon-1}{\tau\varepsilon}\kappa^{2} n^{\kappa}\tilde{j}^{\kappa}\times B^{\kappa}
-\frac{1+\varepsilon}{\tau\varepsilon}\kappa_{ei}K\kappa^{3}\left(n^{\kappa}\right)^{2}\tilde{j}^{\kappa}, \\
\kappa\partial _tE^{\kappa}-\nabla \times B^{\kappa}=-\kappa^{2} n^{\kappa} \tilde{j}^{\kappa},\;\;{\rm div}E^{\kappa}=0,\\
\kappa\partial _tB^{\kappa}+\nabla \times E^{\kappa}=0,\;\;{\rm div}B^{\kappa}=0,
\end{cases}
\end{equation}
which equipped with the initial data
\begin{align}\label{twofiuld nsm2 initial}
\left(n^{\kappa},u^{\kappa},\tilde{j}^{\kappa},E^{\kappa},B^{\kappa}\right)|_{t=0}=\left(n^{\kappa}_{0}(x),u^{\kappa}_{0}(x)
,\tilde{j}^{\kappa}_{0}(x),E^{\kappa}_{0}(x),B^{\kappa}_{0}(x)\right).
\end{align}
Setting $\kappa\rightarrow0$, we formally get the one-fluid compressible Navier-Stokes system:
\begin{equation}\label{twofiuld nsm3}
\begin{cases}
\partial _tn^{0}+\frac{1}{1+\varepsilon}{\rm div}(n^{0}u^{0})=0,\\
\partial _t(n^{0}u^{0})+\frac{1}{1+\varepsilon}{\rm div}(n^{0}u^{0}\otimes u^{0})-\mu\Delta u^{0}-(\mu+\lambda)\nabla {\rm div}u^{0}+\frac{(1+\varepsilon)\eta}{\tau}\nabla P(n^{0})=0,
\end{cases}
\end{equation}
with the initial data
\begin{align}\label{initialn0u0}
\left(n^{0},u^{0}\right)|_{t=0}=\left(n^{0}_{0}(x),u^{0}_{0}(x)\right).
\end{align}
Then letting $\varepsilon\rightarrow0$, we derive the standard compressible Navier-Stokes equation formally. In this paper, we only prove the case of $\kappa\rightarrow0$ rigorously. For $\varepsilon\rightarrow0$, it is similar and much simpler.
For solvability of $\left(n^{0},u^{0}\right)$, we have the following proposition (see \cite{vh}).
\begin{Proposition}\label{solu n0u0}
Let $s>3+\frac{3}{2}$ and assume that the initial data $(n^{0}_{0},u^{0}_{0})$ satisfy
$$n^{0}_{0},\,u^{0}_{0}\in H^{s}(\mathbb{R}^{3}), \;\;\;0<\hat{n}\leq \inf\limits_{x\in\mathbb{R}^{3}}n^{0}_{0}(x)\leq\sup_{x\in\mathbb{R}^{3}}n^{0}_{0}(x)\leq\hat{\hat{n}}<\infty,$$
for some positive constants $\hat{n}$, $\hat{\hat{n}}$. Then there exist constants $T^{\ast}\in (0,\infty)$ and $\tilde{n},\;\tilde{\tilde{n}}>0$ such that the initial value problem \eqref{twofiuld nsm3}--\eqref{initialn0u0} has a unique solution
\begin{align*}
&n^{0}\in C^{k}([0,T^{\ast}], H^{s-k}(\mathbb{R}^{3})),\;\;
u^{0}\in C^{k}([0,T^{\ast}], H^{s-2k}(\mathbb{R}^{3})),\;\;\;k=0,1,\\
&0<\tilde{n}\leq\inf n^{0}(x,t)\leq\sup n^{0}(x,t)\leq\tilde{\tilde{n}}<\infty.
\end{align*}
\end{Proposition}

Now, we are ready to state our main result of this paper.
\begin{Theorem}\label{main result}
Let $l>2+\frac{3}{2}$ and $P(\cdot)$ be a smooth function on $(0,\infty)$ with $P'(\cdot)>0$. Suppose $\left(n^{0},u^{0}\right)$ be the unique classical solution to the equation \eqref{twofiuld nsm3} on $[0,T^{*})$ given in Proposition \ref{solu n0u0}, where $T^{*}\in(0,\infty)$  be the maximal existence time of $\left(n^{0},u^{0}\right)$. Assume the initial data $\left(n^{\kappa}_{0}(x),u^{\kappa}_{0}(x),\tilde{j}^{\kappa}_{0}(x),E^{\kappa}_{0}(x),B^{\kappa}_{0}(x)\right)$ satisfy
$$\left(n^{\kappa}_{0},u^{\kappa}_{0},\kappa\tilde{j}^{\kappa}_{0},E^{\kappa}_{0},B^{\kappa}_{0}\right)\in H^{l}(\mathbb{R}^{3}),\;\;\;{\rm div}B^{\kappa}_{0}=0,\;{\rm div}E^{\kappa}_{0}=0,$$
and
\begin{align}\label{initial}
\left\|\left(n^{\kappa}_{0}-n^{0}_{0},u^{\kappa}_{0}-u^{0}_{0},\kappa\tilde{j}^{\kappa}_{0},E^{\kappa}_{0},B^{\kappa}_{0}\right)\right\|_{l}\leq C_{0}\kappa
\end{align}
for some constant $C_{0}>0$ independent of $\kappa$. Then for any $T_{0}\in (0,T^{*})$, there exist positive constants $\kappa_{0}(T_{0})$ and $\tilde{C}(T_{0})$, such that the system \eqref{twofiuld nsm2} has unique classical solution $\left(n^{\kappa},u^{\kappa},\kappa\tilde{j}^{\kappa},E^{\kappa},B^{\kappa}\right)$ on $[0,T_{0}]$ with
\begin{align}\label{resutest}
\left\|\left(n^{\kappa}-n^{0},u^{\kappa}-u^{0},\kappa\tilde{j}^{\kappa},E^{\kappa},B^{\kappa}\right)(\cdot,t)\right\|_{H^{l}}\leq \tilde{C}\kappa
\end{align}
for any $\kappa\in(0,\kappa_{0}]$ and $t\in[0,T_{0}]$.
\end{Theorem}
This paper is organized as follows. In Section \ref{sec1}, we derive the error system and introduce the local existence of the unique classical solution. In section \ref{sec2} and section \ref{sec3}, we derive a uniform decay estimate with respect to $\kappa$ of the error system \eqref{2.1}--\eqref{2.6} by employing the energy method. In the last section, based on Lemma \ref{lem2}, we deduce Theorem \ref{main result} by the bootstrap principle.

\section{Derivation of the error system and local existence}\label{sec1}
Let $\left(n^{\kappa},u^{\kappa},\kappa\tilde{j}^{\kappa},E^{\kappa},B^{\kappa}\right)$ be the solution to system \eqref{twofiuld nsm2}--\eqref{twofiuld nsm2 initial} and $\left(n^{0},u^{0}\right)$ be the solution to the system \eqref{twofiuld nsm3}--\eqref{initialn0u0} given in Proposition \ref{solu n0u0}.
Set
$$N^{\kappa}=n^{\kappa}-n^{0},\;U^{\kappa}=u^{\kappa}-u^{0}.$$
For $\rho>0$, define
$$h(\rho)=\int_{1}^{\rho}\frac{P'(s)}{s}ds.$$ Combining the equations \eqref{twofiuld nsm2}--\eqref{initialn0u0}, we obtain the error system:\small{
\begin{align}\label{2.1}
\partial _tN^{\kappa}+\frac{1}{1+\varepsilon}\left({\rm div}\left(\left(N^{\kappa}+n^{0}\right)U^{\kappa}\right)+{\rm div}\left(N^{\kappa}u^{0}\right)\right)=0,
\end{align}
\begin{align}\label{2.2}
\partial _tU^{\kappa}&
+\frac{\varepsilon}{1+\varepsilon}\kappa^{2}\tilde{j}^{\kappa}\cdot\nabla\tilde{j}^{\kappa}
+\frac{1}{1+\varepsilon}\left(\left(U^{\kappa}+u^{0}\right)\cdot\nabla U^{\kappa}+U^{\kappa}\cdot \nabla u^{0}\right)\notag\\
&\quad-\frac{\mu}{N^{\kappa}+n^{0}}\Delta U^{\kappa}-\frac{\mu+\lambda}{N^{\kappa}+n^{0}}\nabla{\rm div} U^{\kappa}+\frac{\eta(1+\varepsilon)}{\tau}\nabla\left(h(N^{\kappa}+n^{0})-h(n^{0})\right)\notag\\
&=\left(\frac{1}{N^{\kappa}+n^{0}}-\frac{1}{n^{0}}\right)\left(\mu\Delta u^{0}+(\mu+\lambda)\nabla{\rm div}u^{0}\right)+\frac{\kappa^{2}}{\tau}\tilde{j}^{\kappa}\times B^{\kappa},
\end{align}
\begin{align}\label{2.3}
\kappa\partial_t\tilde{j}^{\kappa}&+\frac{\kappa}{1+\varepsilon}
\left(\left(U^{\kappa}+u^{0}\right)\cdot\nabla\tilde{j}^{\kappa}+\tilde{j}^{\kappa}\cdot\nabla\left(U^{\kappa}+u^{0}\right)\right)
+\frac{\varepsilon-1}{\varepsilon+1}\kappa^{2}\tilde{j}^{\kappa}\cdot\nabla\tilde{j}^{\kappa}\notag\\
&\quad\!\!-\mu\frac{\kappa}{N^{\kappa}+n^{0}}\Delta\tilde{j}^{\kappa}-(\mu+\lambda)\frac{\kappa}{N^{\kappa}+n^{0}}\nabla{\rm div}\tilde{j}^{\kappa}\notag\\
&=\frac{1+\varepsilon}{\tau\varepsilon}E^{\kappa}+\frac{\kappa}{\tau\varepsilon}(U^{\kappa}+u^{0})\times B^{\kappa}+\frac{\varepsilon-1}{\tau\varepsilon}
\kappa^{2}\tilde{j}^{\kappa}\times B^{\kappa}-\frac{\varepsilon+1}{\tau\varepsilon}\kappa_{ei}K\kappa^{3}\left(N^{\kappa}+n^{0}\right)\tilde{j}^{\kappa},
\end{align}
\begin{flalign}\label{2.4}
\partial _tE^{\kappa}-\frac{1}{\kappa}\nabla \times B^{\kappa}=-\kappa(N^{\kappa}+n^{0}) \tilde{j}^{\kappa},\;\;{\rm div}E^{\kappa}=0,
\end{flalign}
\begin{gather}\label{2.5}
\partial _tB^{\kappa}+\frac{1}{\kappa}\nabla \times E^{\kappa}=0,\;\;{\rm div}B^{\kappa}=0,
\end{gather}
\begin{align}\label{2.6}
\left(N^{\kappa},U^{\kappa},\tilde{j}^{\kappa},E^{\kappa},B^{\kappa}\right)|_{t=0}&=\left(n^{\kappa}_{0}-n^{0}_{0},u^{\kappa}_{0}-u^{0}_{0},
\tilde{j}^{\kappa}_{0},E^{\kappa}_{0},B^{\kappa}_{0}\right)\notag\\
&=:\left(N^{\kappa}_{0},U^{\kappa}_{0},\tilde{j}^{\kappa}_{0},E^{\kappa}_{0},B^{\kappa}_{0}\right).
\end{align}}
Let
$$W_{1}^{\kappa}=\left(
\begin{array}{cc}
U^{\kappa} \\
\kappa\tilde{j}^{\kappa} \\
\end{array}\right),\;\;
W_{2}^{\kappa}=\left(
\begin{array}{cc}
N^{\kappa} \\
E^{\kappa} \\
B^{\kappa}\\
\end{array}\right),\;\;
W^{\kappa}=\left(\begin{array}{cc}
W_{1}^{\kappa}\\
W_{2}^{\kappa}\\
\end{array}\right),$$
$$A_{i}^{\kappa}=\left(
\begin{array}{ccc}
-\frac{1}{1+\varepsilon}\left(u^{0}_{i}+U^{\kappa}_{i}\right)& 0& 0  \\
0 & 0 & \frac{1}{\kappa}B_{i}\\
0 & \frac{1}{\kappa}B_{i}^{T} & 0\\
\end{array}\right),
$$
$$
A_{ij}^{\kappa}=\left(
\begin{array}{cc}
\frac{\mu}{N^{\kappa}+n^{0}}e_{i}^{T}e_{j}I_{3\times3}+\frac{\mu+\lambda}{N^{\kappa}+n^{0}}e_{i}e_{j}^{T}& 0  \\
0 & \frac{\mu}{N^{\kappa}+n^{0}}e_{i}^{T}e_{j}I_{3\times3}+\frac{\mu+\lambda}{N^{\kappa}+n^{0}}e_{i}e_{j}^{T}\\
\end{array}\right),
$$
where
$$B_{1}=\left(
\begin{array}{ccc}
0&0&0\\
0&0&-1\\
0&1&0\\
\end{array}
\right),\;
B_{2}=\left(
\begin{array}{ccc}
0&0&1\\
0&0&0\\
-1&0&0\\
\end{array}
\right),\;
B_{3}=\left(
\begin{array}{ccc}
0&-1&0\\
1&0&0\\
0&0&0\\
\end{array}
\right),\;
e_i=
\begin{array}{cc}
\left(\begin{array}{c}0\\
\vdots\\
1\\
\vdots\\
0 \end{array}\right)
\begin{array}{c}\end{array}.
\end{array}
$$

Then the error system \eqref{2.1}--\eqref{2.6} can be recast as
\begin{equation}\label{matrix}
\begin{cases}
\partial_{t}W_{1}^{\kappa}=\sum\limits_{i,j=1}^{3}A_{ij}^{\kappa}\partial_{x_{i}x_{j}}W_{1}^{\kappa}+F\left(\partial^{\beta}W^{\kappa},\partial^{\beta}n^{0},
\partial^{\alpha}u^{0}\right),\\
\partial_{t}W_{2}^{\kappa}=\sum\limits_{i=1}^{3}A_{i}^{\kappa}\partial_{x_{i}}W_{2}^{\kappa}+G\left(W_{2}^{\kappa}, \partial^{\beta}W_{1}^{\kappa},
\partial^{\beta}n^{0},\partial^{\beta}u^{0}\right),\\
W^{\kappa}|_{t=0}=W^{\kappa}_{0},
\end{cases}
\end{equation}
for some functions $F$ and $G$, and the multi-index $\alpha$ and $\beta$ satisfy $|\alpha|\leq2$, $|\beta|\leq1$.
Noted that the first group of equations is in the form of parabolic system with respect to $W_{1}^{\kappa}$ while the second group of equations is symmetric hyperbolic system with respect to $W_{2}^{\kappa}$. Thus, by the theory in \cite{vh}, we have the following result for initial value problem \eqref{matrix} (or \eqref{2.1}--\eqref{2.6}).
\begin{Proposition}\label{solu nk}
Let $l>2+\frac{3}{2}$ and assume that $\left(n^{0},u^{0}\right)$ satisfy the conditions in Proposition \ref{solu n0u0}. Moreover, suppose that the initial data satisfy $$\left(N^{\kappa}_{0},U^{\kappa}_{0},\kappa\tilde{j}^{\kappa}_{0},E^{\kappa}_{0},B^{\kappa}_{0}\right)\in H^{l}(\mathbb{R}^{3}),\;\;
\inf\limits_{x\in \mathbb{R}^3}N^{\kappa}_{0}>0,\;\;\left\|N^{\kappa}_{0}\right\|_{l}\leq\delta$$
for some small positive constant $\delta$. Then there exist positive constants $T_{\kappa}(0<T_{\kappa}<\infty)$ and $M$ ($M$ only depends on $\delta$) such that the initial data problem \eqref{matrix} has a unique classical solution $\left(N^{\kappa},U^{\kappa},\tilde{j}^{\kappa},E^{\kappa},B^{\kappa}\right)$ satisfying $\left\|N^{\kappa}(t)\right\|_{l}\leq M\delta$, $\inf\limits_{x\in \mathbb{R}^3}N^{\kappa}>0$ and
$$\left(N^{\kappa},E^{\kappa},B^{\kappa}\right)\in C^{k}([0,T_{\kappa}], H^{l-k}(\mathbb{R}^{3})),\;\;
\left(U^{\kappa},\kappa\tilde{j}^{\kappa}\right)\in C^{k}([0,T_{\kappa}], H^{l-2k}(\mathbb{R}^{3})),\;\;\;k=0,1.$$
\end{Proposition}

We will show that for any $T_{0}<T^{*}$, there exists $\kappa_{0}>0$, such that the existence time $T_{\kappa}>T_{0}$ for any $0<\kappa<\kappa_{0}$. For this purpose, one should establish the uniform estimates on $\left(N^{\kappa},U^{\kappa},\kappa\tilde{j}^{\kappa},E^{\kappa},B^{\kappa}\right)$.

\section{Proof of the main result}
For any $0<T_{1}<1$ independent of $\kappa$, let $T=T^{\kappa}={\rm min}\{T_{1},T_{\kappa}\}$. The positive constant $C$ depends upon $C_{0}$ and $T_{0}$ with $C>C_{0}$. For the sake of convenience, we define
$$\|W^{\kappa}(t)\|^{2}_{l}=\|\left(N^{\kappa},U^{\kappa},\kappa\tilde{j}^{\kappa},E^{\kappa},B^{\kappa}\right)(t)\|^{2}_{l}.$$
\subsection{Zero-order estimates}\label{sec2}
\begin{Lemma}\label{lem1} Under the hypothesis in Theorem \ref{main result}, for any $t\in (0,T)$ and sufficiently small $\kappa$, we have
\small{
\begin{align}
&\left(\|U^{\kappa}\|^{2}+\left\|\kappa\tilde{j}^{\kappa}\right\|^{2}+\|E^{\kappa}\|^{2}+\|G^{\kappa}\|^{2}\right)(t)+\int\int_{0}^{N^{\kappa}}
h\left(s+n^{0}\right)-h\left(n^{0}\right)dsdx\notag\\
&\quad+\int_{0}^{t}\mu\left(\left\|\nabla U^{\kappa}\right\|^{2}+\left\|\nabla \left(\kappa\tilde{j}^{\kappa}\right)\right\|^{2}\right)(\tau)
+(\mu+\lambda)\left(\left\|{\rm div} U^{\kappa}\right\|^{2}+\left\|{\rm div}\left(\kappa\tilde{j}^{\kappa}\right)\right\|^{2}\right)(\tau)d\tau\notag\\
&\leq C\left(\|U^{\kappa}\|^{2}+\left\|\kappa\tilde{j}^{\kappa}\right\|^{2}+\|E^{\kappa}\|^{2}+\|G^{\kappa}\|^{2}+
\int\int_{0}^{N^{\kappa}}h\left(s+n^{0}\right)-h\left(n^{0}\right)dsdx\right)(t=0)\notag\\
&\quad+C\int_{0}^{t}\left(\|W^{\kappa}\|^{2}_{l}+\|W^{\kappa}\|^{3}_{l}+\|W^{\kappa}\|^{4}_{l}\right)(\tau)d\tau.
\end{align}}
\end{Lemma}
\begin{proof}
Multiplying \eqref{2.2} by $(N^\kappa+n^0)U^k$, and integrating over the whole space, we obtain
\begin{align}\label{Ul2}
&\dfrac{1}{2}\dfrac{d}{dt}\int\left(N^\kappa+n^0\right){|U^\kappa|}^2dx+\mu{\|\nabla{U^\kappa}\|}^2+(\mu+\lambda){\|{\rm div}{U^\kappa}\|}^2\notag\\
&=\dfrac{\eta(1+\varepsilon)}{\tau}\left(h\left(N^\kappa+n^0\right)-h\left(n^0\right),{\rm div}\left(\left(N^\kappa+n^0\right)U^\kappa\right)\right)+\dfrac{1}{2}\int\partial_t\left({N^\kappa}
+{n^0}\right){|U^\kappa|}^2dx\notag\\
&\quad-\dfrac{1}{1+\varepsilon}\left(\left(U^\kappa+u^0\right)\cdot\nabla{U^\kappa}+U^\kappa\cdot\nabla{u^0},\left(N^\kappa+n^0\right)U^\kappa\right)\notag\\
&\quad-\dfrac{\varepsilon}{1+\varepsilon}\kappa^2\left(\tilde{j}^{\kappa}\cdot\nabla\tilde{j}^{\kappa},\left(N^\kappa+n^0\right)U^\kappa\right)+\dfrac{\kappa^2}{\tau}
\left(\tilde{j}^{\kappa}\times B^\kappa,\left(N^\kappa+n^0\right)U^\kappa\right)\notag\\
&\quad+\int\left(\dfrac{1}{N^\kappa+n^0}-\dfrac{1}{n^0}\right)\left(\mu\Delta{u^0}
+(\mu+\lambda)\nabla{{\rm div}u^0}\right)\left(N^\kappa+n^0\right)U^\kappa dx.
\end{align}
Firstly, we deal with the second term on the right-hand side of \eqref{Ul2}. Noted that
\begin{align}\label{3.2'}\partial_t\left({N^\kappa}+{n^0}\right)+\frac{1}{1+\varepsilon}{\rm div}\left(\left(N^\kappa+n^0\right)\left(U^\kappa+u^0\right)\right)=0,
\end{align}
in view of the regularity of $(n^{0},u^{0})$, H{\"o}lder's inequality and Sobolev's imbedding theorem, we have
\begin{align}\label{3.2}
\begin{split}
\dfrac{1}{2}\int\partial_t\left({N^\kappa}+{n^0}\right){|U^\kappa|}^2dx&=-\frac{1}{2(1+\varepsilon)}\int{\rm div}\left(\left(N^\kappa+n^0\right)\left(U^\kappa+u^0\right)\right){|U^\kappa|}^2dx\\
&\lesssim\left(1+\|N^\kappa\|_l+\|U^\kappa\|_l+\|N^\kappa\|_l\|U^\kappa\|_l\right){\|U^\kappa\|}^2.
\end{split}
\end{align}
For the last four terms on the right-hand side of \eqref{Ul2}, by using the regularity of $(n^{0},u^{0})$, H{\"o}lder's inequality, Sobolev's imbedding theorem and the bounds on $N^{\kappa}$ and $n^{0}$ stated in Proposition \ref{solu n0u0} and Proposition \ref{solu nk}, we have
\begin{align}
\left|-\dfrac{1}{1+\varepsilon}\left(\left(U^\kappa+u^0\right)\cdot\nabla{U^\kappa}+U^\kappa\cdot\nabla{u^0},\left(N^\kappa+n^0\right)U^\kappa\right)\right|
\lesssim\left(1+\|N^\kappa\|_l+\|U^\kappa\|_l\right){\|U^\kappa\|_l}^2,
\end{align}
\begin{align}
\left|-\dfrac{\varepsilon}{1+\varepsilon}\kappa^2\left(\tilde{j}^{\kappa}\cdot\nabla\tilde{j}^{\kappa},\left(N^\kappa+n^0\right)U^\kappa\right)\right|
\lesssim(1+\|N^\kappa\|_l)\|U^\kappa\|_l\left\|\kappa\tilde{j}^{\kappa}\right\|_l^2,
\end{align}
\begin{align}
\left|\dfrac{\kappa^2}{\tau}\left(\tilde{j}^{\kappa}\times B^\kappa,\left(N^\kappa+n^0\right)U^\kappa\right)\right|
\lesssim\kappa\left\|\kappa{\tilde{j}}^\kappa\right\|_l\|U^\kappa\|_l\|B^\kappa\|_l(1+\|N^\kappa\|_l),
\end{align}
\begin{align}
\begin{split}
&\left|\int\left(\dfrac{1}{N^\kappa+n^0}-\dfrac{1}{n^0}\right)\left(\mu\Delta{u^0}+(\mu+\lambda)\nabla{{\rm div}u^0}\right)\left(N^\kappa+n^0\right)U^\kappa dx\right|\\
&\quad\lesssim(1+\|N^\kappa\|)\|U^\kappa\|\|N^\kappa\|_l.
\end{split}
\end{align}
Now, we deal with the first term on the right-hand side of \eqref{Ul2}. From equation \eqref{2.1}, we deduce
\begin{align}\label{3.3}
&\dfrac{\eta(1+\varepsilon)}{\tau}\int\left(h\left(N^\kappa+n^0\right)-h\left(n^0\right)\right)\left(-(1+\varepsilon)\partial_tN^\kappa-{\rm div}\left(N^\kappa u^0\right)\right)dx\notag\\
&=-\frac{\eta(1+\varepsilon)^2}{\tau}\left(\frac{d}{dt}\int\!\!\int_{0}^{N^{\kappa}} h\left(s+n^0\right)-h\left(n^0\right)dsdx
-\int\!\!\int_{0}^{N^{\kappa}}\left(h'\left(s+n^0\right)-h'\left(n^0\right)\right)\partial _{t}n^0dsdx\right)\notag\\
&\quad-\frac{\eta(1+\varepsilon)}{\tau}\int\left(h\left(N^\kappa+n^0\right)-h\left(n^0\right)\right){\rm div}\left(N^\kappa u^0\right)dx\notag\\
&\leq-\frac{\eta(1+\varepsilon)^2}{\tau}\frac{d}{dt}\int \int_{0}^{N^{\kappa}}h\left(s+n^0\right)-h\left(n^0\right)dsdx+C\parallel N^\kappa\parallel^{2}_{l},
\end{align}
where the last inequality is deduced by the regularity of $(n^{0},u^{0})$ and the upper bounds on $N^{\kappa}$ and $n^0$.
Inserting \eqref{3.2}--\eqref{3.3} into \eqref{Ul2}, we have
\begin{align}\label{es01}
\begin{split}
&\frac{1}{2}\frac{d}{dt}\int \left(N^\kappa+n^0\right)\left| U^{\kappa}\right|^{2}dx+\frac{\eta(1+\varepsilon)^2}{\tau}\frac{d}{dt}\int\int_{0}^{N^{\kappa}} h\left(s+n^0\right)-h\left(n^0\right)dsdx\\
&\quad\!\!+\mu\left\| \nabla U^{\kappa}\right\|^{2}+(\mu+\lambda)\parallel {\rm div}U^{\kappa}\parallel^{2}\\
&\lesssim\|W^{\kappa}\|^{2}_{l}+\|W^{\kappa}\|^{3}_{l}+\|W^{\kappa}\|^{4}_{l}.
\end{split}
\end{align}

Similarly, multiplying \eqref{2.3} by $(N^\kappa+n^0)\kappa\tilde{j}^k$,  and then integrating over the whole space, we have
\begin{align}\label{3.11}
\begin{split}
&\frac{1}{2}\frac{d}{dt}\int\left|\kappa\tilde{j}^{\kappa}\right|^{2}(N^\kappa+n^0)dx
+\mu\left\|\nabla\left(\kappa\tilde{j}^{\kappa}\right)\right\|^{2}+(\mu+\lambda)\left\| {\rm div}\left(\kappa\tilde{j}^{\kappa}\right)\right\|^{2}\\
&=-\frac{\kappa^2}{1+\varepsilon}\left(\left(U^{\kappa}+u^0\right)\cdot\nabla \tilde{j}^{\kappa}+\tilde{j}^{\kappa}\cdot\nabla\left(U^{\kappa}+u^0\right),\left(N^{\kappa}+n^0\right)\tilde{j}^{\kappa}\right)\\
&\quad-\frac{\varepsilon-1}{\varepsilon+1}\kappa^3\left(\tilde{j}^{\kappa}\cdot\nabla\tilde{j}^{\kappa},\left(N^{\kappa}+n^0\right)\tilde{j}^{\kappa}\right)
+\frac{\varepsilon+1}{\varepsilon\tau}\left(E^{\kappa},\kappa\left(N^{\kappa}+n^0\right)\tilde{j}^{\kappa}\right)\\
&\quad+\frac{\kappa^2}{\varepsilon\tau}\left(\left(U^{\kappa}+u^0\right)\times B^{\kappa},\left(N^{\kappa}
+n^0\right)\tilde{j}^{\kappa}\right)
+\frac{\varepsilon-1}{\varepsilon\tau}\kappa^3\left(\tilde{j}^{\kappa}\times B^{\kappa},\left(N^{\kappa}+n^0\right)\tilde{j}^{\kappa}\right)\\
&\quad-\frac{\varepsilon+1}{\varepsilon\tau}\kappa_{e_{i}}K\kappa^4\left(\left(N^{\kappa}+n^0\right)\tilde{j}^{\kappa},\left(N^{\kappa}+n^0\right)\tilde{j}^{\kappa}\right)
+\frac{1}{2}\int\left| \kappa\tilde{j}^{\kappa}\right|^{2}\partial_{t}\left(N^{\kappa}+n^0\right)dx.
\end{split}
\end{align}
Applying the regularity of $(n^{0},u^{0})$, H{\"o}lder's inequality and Sobolev's imbedding theorem, the terms on the right-hand side of \eqref{3.11} can be estimated as:
\begin{align}\label{3.12}
\begin{split}
&\left|-\frac{\kappa^2}{1+\varepsilon}\left(\left(U^{\kappa}+u^0\right)\cdot\nabla \tilde{j}^{\kappa}+\tilde{j}^{\kappa}\cdot\nabla\left(U^{\kappa}+u^0\right),\left(N^{\kappa}+n^0\right)\tilde{j}^{\kappa}\right)\right|\\
&\quad\lesssim\left\|\kappa\tilde{j}^{\kappa}\right\|^{2}_{l}\left(1+\left\|N^{\kappa}\right\|_{l}+\|U^{\kappa}\|_{l}+\|N^{\kappa}\|_{l}\|U^{\kappa}\|_{l}\right),
\end{split}
\end{align}
\begin{align}
\left|-\frac{\varepsilon-1}{\varepsilon+1}\kappa^3\left(\tilde{j}^{\kappa}\cdot\nabla\tilde{j}^{\kappa},\left(N^{\kappa}+n^0\right)\tilde{j}^{\kappa}\right)\right|
\lesssim\left\|\kappa\tilde{j}^{\kappa}\right\|^{3}_{l}\left(1+\left\|N^{\kappa}\right\|_{l}\right),
\end{align}
\begin{align}
\left|\frac{\varepsilon+1}{\varepsilon\tau}\left(E^{\kappa},\kappa\left(N^{\kappa}+n^0\right)\tilde{j}^{\kappa}\right)\right|
\lesssim\left\|\kappa\tilde{j}^{\kappa}\right\|_{l}\left\|E^{\kappa}\right\|_{l}\left(1+\left\|N^{\kappa}\right\|_{l}\right),
\end{align}
\begin{align}
\begin{split}
&\left|\frac{\kappa^2}{\varepsilon\tau}\left(\left(U^{\kappa}+u^0\right)\times B^{\kappa},\left(N^{\kappa}+n^0\right)\tilde{j}^{\kappa}\right)\right|\\
&\quad\!\!\lesssim\kappa\left\|\kappa\tilde{j}^{\kappa}\right\|_{l}\left\|B^{\kappa}\right\|_{l}\left(1+\left\|N^{\kappa}\right\|_{l}+\|U^{\kappa}\|_{l}+\|N^{\kappa}\|_{l}\|U^{\kappa}\|_{l}\right),
\end{split}
\end{align}
\begin{align}
\left|\frac{\varepsilon-1}{\varepsilon\tau}\kappa^3\left(\tilde{j}^{\kappa}\times B^{\kappa},\left(N^{\kappa}+n^0\right)\tilde{j}^{\kappa}\right)\right|
\lesssim\kappa\left\|\kappa\tilde{j}^{\kappa}\right\|^2_{l}\left\|B^{\kappa}\right\|_{l}\left(1+\left\|N^{\kappa}\right\|_{l}\right),
\end{align}
\begin{align}
\left|-\frac{\varepsilon+1}{\varepsilon\tau}\kappa_{e_{i}}K\kappa^4\left(\left(N^{\kappa}+n^0\right)\tilde{j}^{\kappa},\left(N^{\kappa}+n^0\right)\tilde{j}^{\kappa}\right)\right|
\lesssim\kappa^{2}\left\|\kappa\tilde{j}^{\kappa}\right\|^{2}_{l}\left(1+\|N^{\kappa}\|_{l}+\|N^{\kappa}\|^{2}_{l}\right),
\end{align}
\begin{align}\label{3.18}
\begin{split}
\left|\frac{1}{2}\int\left| \kappa\tilde{j}^{\kappa}\right|^{2}\partial_{t}\left(N^{\kappa}+n^0\right)dx\right|&=\left|-\frac{1}{2(1+\varepsilon)}\int\left| \kappa\tilde{j}^{\kappa}\right|^{2}{\rm div}\left(\left(N^\kappa+n^0\right)\left(U^\kappa+u^0\right)\right)dx\right|\\
&\lesssim\left\|\kappa\tilde{j}^{\kappa}\right\|^{2}_{l}\left(1+\left\|N^{\kappa}\right\|_{l}+\|U^{\kappa}\|_{l}+\|N^{\kappa}\|_{l}\|U^{\kappa}\|_{l}\right),
\end{split}
\end{align}
where we have used \eqref{3.2'} in the last equality. Substituting \eqref{3.12}--\eqref{3.18} into \eqref{3.11}, we conclude that
\begin{align}\label{es02}
\begin{split}
&\frac{1}{2}\frac{d}{dt}\int \left(N^\kappa+n^0\right)\left| \kappa\tilde{j}^{\kappa}\right|^{2}dx+\mu\left\|\nabla \left(\kappa\tilde{j}^{\kappa}\right)\right\|^{2}+(\mu+\lambda)\left\|{\rm div}\left(\kappa\tilde{j}^{\kappa}\right)\right\|^{2}\\
&\lesssim\|W^{\kappa}\|^{2}_{l}+\|W^{\kappa}\|^{3}_{l}+\|W^{\kappa}\|^{4}_{l}.
\end{split}
\end{align}

Then, multiplying \eqref{2.4}, \eqref{2.5} by $E^{\kappa}$ and $G^{\kappa}$ respectively, and then integrating over the whole space, one has
\begin{align}\label{es03}
\begin{split}
\frac{1}{2}\frac{d}{dt}\left(\left\|E^{\kappa}\right\|^{2}+\left\|G^{\kappa}\right\|^{2}\right)
&=-\left(\kappa\left(N^{\kappa}+n^{0}\right)\tilde{j}^{\kappa},E^{\kappa}\right)\\
&\lesssim\left\|\kappa\tilde{j}^{\kappa}\right\|_{l}\|E^{\kappa}\|_{l}\left(1+\|N^{\kappa}\|_{l}\right)\\
&\lesssim \|W^{\kappa}\|^{2}_{l}+\left\|W^{\kappa}\right\|^{3}_{l},
\end{split}
\end{align}
where we have used the regularity of $n^{0}$, H{\"o}lder's inequality and Sobolev's imbedding theorem.

Therefore, combining \eqref{es01}, \eqref{es02} and \eqref{es03}, we have
\begin{align}
\begin{split}
&\frac{1}{2}\frac{d}{dt}\int\left(N^{\kappa}+n^{0}\right)\left(|U^{\kappa}|^{2}+\left|\kappa\tilde{j}^{\kappa}\right|^{2}\right)dx
+\frac{1}{2}\frac{d}{dt}\left(\|E^{\kappa}\|^{2}+\|G^{\kappa}\|^{2}\right)\\
&\quad\!\!+\frac{\eta\left(1+\varepsilon\right)}{\tau}\frac{d}{dt}\int\int^{N^{\kappa}}_{0} h\left(s+n^{0}\right)-h\left(n_{0}\right)dsdx
+\mu\|\nabla U^{\kappa}\|^{2}+(\mu+\lambda)\|{\rm div}U^{\kappa}\|^{2}\\
&\quad\!\!+\mu\left\|\nabla\left(\kappa\tilde{j}^{\kappa}\right)\right\|^{2}+(\mu+\lambda)\left\|{\rm div}\left(\kappa\tilde{j}^{\kappa}\right)\right\|^{2}\\
&\lesssim\|W^{\kappa}\|^{2}_{l}+\|W^{\kappa}\|^{3}_{l}+\|W^{\kappa}\|^{4}_{l}.
\end{split}
\end{align}
By employing the fact that $N^{\kappa}+n^{0}>\tilde{n}>0$, which is deduced by $\inf\limits_{x\in \mathbb{R}^3}N^{\kappa}>0$, then we conclude \eqref{lem1}.
\end{proof}

\subsection{Higher order estimates}\label{sec3}
\begin{Lemma}\label{lem2}Under the hypothesis in Theorem \ref{main result}, for any $t\in (0,T)$ and sufficiently small $\kappa$, we have
\begin{align}
&\left(\left\|U^{\kappa}\right\|_{l}^{2}+\left\|\kappa\tilde{j}^{\kappa}\right\|_{l}^{2}+\left\|E^{\kappa}\right\|_{l}^{2}+\left\|G^{\kappa}\right\|_{l}^{2}\right)(t)
+\int_{0}^{t}\left(\left\|U^{\kappa}\right\|_{l+1}^{2}+\left\|\kappa\tilde{j}^{\kappa}\right\|_{l+1}^{2}\right)(\tau)d\tau\notag\\
&\quad+\int\int_{0}^{N^{\kappa}}h\left(s+n^{0}\right)-h\left(n^{0}\right)dsdx+\sum\limits_{1\leq\alpha\leq l}\int\frac{h'\left(N^{\kappa}+n^{0}\right)}{N^{\kappa}+n^{0}}\left|\partial^{\alpha}N^{\kappa}\right|^{2}dx\notag\\
&\leq\left(\left\|U^{\kappa}\right\|_{l}^{2}+\left\|\kappa\tilde{j}^{\kappa}\right\|_{l}^{2}+\left\|E^{\kappa}\right\|_{l}^{2}+\left\|G^{\kappa}\right\|_{l}^{2}\right)(t=0)
+\left(\sum\limits_{1\leq\alpha\leq l}\int\frac{h'\left(N^{\kappa}+n^{0}\right)}{N^{\kappa}+n^{0}}\left|\partial^{\alpha}N^{\kappa}\right|^{2}dx\right)(t=0)\notag\\
&\quad+\left(\int\int_{0}^{N^{\kappa}}h\left(s+n^{0}\right)-h\left(n^{0}\right)dsdx\right)(t=0)+C\int_{0}^{t}\sum\limits_{i=2}^{6}\|W^{\kappa}\|^{i}_{l}(\tau)d\tau
\end{align}
\end{Lemma}

\begin{proof}
Set $1\leq\alpha\leq l$. Applying the operator $\partial^{\alpha}$ to \eqref{2.2}, and taking the inner product of the resulting equation and $\partial^{\alpha}U^{\kappa}$, we obtain
\begin{align}\label{hoe}
\frac{1}{2}\frac{d}{dt}\left\|\partial^{\alpha}U^{\kappa}\right\|^{2}
&=\mu\left(\partial^{\alpha}\left(\frac{\Delta U^{\kappa}}{N^{\kappa}+n^{0}}\right), \partial^{\alpha}U^{\kappa}\right)+(\mu+\lambda)\left(\partial^{\alpha}\left(\frac{\nabla{\rm div} U^{\kappa}}{N^{\kappa}+n^{0}}\right), \partial^{\alpha}U^{\kappa}\right)\notag\\
&\quad-\frac{\eta(1+\varepsilon)}{\tau}\left(\partial^{\alpha}\nabla\left(h(N^{\kappa}+n^{0})-h(n^{0})\right), \partial^{\alpha}U^{\kappa}\right)\notag\\
&\quad-\frac{1}{1+\varepsilon}\left(\partial^{\alpha}\left((U^{\kappa}+u^{0})\cdot\nabla U^{\kappa}\right),\partial^{\alpha}U^{\kappa}\right)
-\frac{1}{1+\varepsilon}\left(\partial^{\alpha}\left(U^{\kappa}\cdot\nabla u^{0}\right),\partial^{\alpha}U^{\kappa}\right)\notag\\
&\quad-\frac{\varepsilon}{1+\varepsilon}\kappa^{2}\left(\partial^{\alpha}\left(\tilde{j}^{\kappa}\cdot\nabla \tilde{j}^{\kappa}\right),\partial^{\alpha}U^{\kappa}\right)\notag\\
&\quad+\mu\left(\partial^{\alpha}\left(
\left(\frac{1}{N^{\kappa}+n^{0}}-\frac{1}{n^{0}}\right)\Delta u^{0}\right), \partial^{\alpha}U^{\kappa}\right)\notag\\
&\quad+(\mu+\lambda)\left(\partial^{\alpha}\left(
\left(\frac{1}{N^{\kappa}+n^{0}}-\frac{1}{n^{0}}\right)\nabla{\rm div} u^{0}\right), \partial^{\alpha}U^{\kappa}\right)\notag\\
&\quad+\frac{\kappa^{2}}{\tau}\left(\partial^{\alpha}\left(\tilde{j}^{\kappa}\times B^{\kappa}
\right), \partial^{\alpha}U^{\kappa}\right)\notag\\
&=:\sum\limits_{i=1}^{9}I^{(i)}.
\end{align}
Next, we will handel with all terms on the right-hand side of \eqref{hoe} one by one. For the term $I^{(1)}$, integrating by parts and using the fact that $N^{\kappa}+n^{0}>\tilde{n}>0$, we have
\begin{align*}
I^{(1)}&=\mu\left(\frac{\partial^{\alpha}\Delta U^{\kappa}}{N^{\kappa}+n^{0}}, \partial^{\alpha}U^{\kappa}\right)+\mu\left(\mathcal{H}^{(1)}_{U}, \partial^{\alpha}U^{\kappa}\right)\\
&=-\mu\int\frac{1}{N^{\kappa}+n^{0}}\left|\partial^{\alpha}\nabla U^{\kappa}\right|^{2}dx-\mu\left(\partial^{\alpha}\nabla U^{\kappa},\partial^{\alpha} U^{\kappa}\nabla\left(\frac{1}{N^{\kappa}+n^{0}}\right)\right)+\mu\left(\mathcal{H}^{(1)}_{U}, \partial^{\alpha}U^{\kappa}\right)\\
&\leq-\mu\int\frac{1}{N^{\kappa}+n^{0}}\left|\partial^{\alpha}\nabla U^{\kappa}\right|^{2}dx+
C\left\|U^{\kappa}\right\|_{l+1}\left\|U^{\kappa}\right\|_{l}\left(1+\left\|N^{\kappa}\right\|_{l}\right)+\mu\left\|\mathcal{H}^{(1)}_{U}\right\|\|U^{\kappa}\|_l,
\end{align*}
where
$$\mathcal{H}^{(1)}_{U}=\partial^{\alpha}\left(\frac{\Delta U^{\kappa}}{N^{\kappa}+n^{0}}\right)-\frac{\partial^{\alpha}\Delta U^{\kappa}}{N^{\kappa}+n^{0}}.$$
Applying the Moser-type inequalities in Lemma \ref{lem1.1}, one has
\begin{align*}
\left\|\mathcal{H}^{(1)}_{U}\right\|
&\lesssim\left\|D\left(\frac{1}{N^{\kappa}+n^{0}}\right)\right\|_{\infty}\left\|D^{l-1}\Delta U^{\kappa}\right\|
+\left\|\Delta U^{\kappa}\right\|_{\infty}\left\|D^{l}\left(\frac{1}{N^{\kappa}+n^{0}}\right)\right\|\\
&\lesssim\left(1+\left\|N^{\kappa}\right\|_{l}\right)\left\|U^{\kappa}\right\|_{l+1}.
\end{align*}
Thus, one deduces that
\begin{align}\label{i1}
I^{(1)}\leq-\mu\int\frac{1}{N^{\kappa}+n^{0}}\left|\partial^{\alpha}\nabla U^{\kappa}\right|^{2}dx+C\left\|U^{\kappa}\right\|_{l+1}\left\|U^{\kappa}\right\|_{l}\left(1+\left\|N^{\kappa}\right\|_{l}\right).
\end{align}

Similarly, for the term $I^{(2)}$, we have
\begin{align}\label{i2}
I^{(2)}\leq-(\mu+\lambda)\int\frac{1}{N^{\kappa}+n^{0}}\left|\partial^{\alpha}{\rm div} U^{\kappa}\right|^{2}dx+C\left\|U^{\kappa}\right\|_{l+1}\left\|U^{\kappa}\right\|_{l}\left(1+\left\|N^{\kappa}\right\|_{l}\right).
\end{align}

Next, we will deal with the term $I^{(3)}$. Integrating by parts and applying Leibniz's formula, we obtain
\begin{align}\label{i3bds}
I^{(3)}
&=\frac{\eta(1+\varepsilon)}{\tau}\left(\partial^{\alpha}\left(h(N^{\kappa}+n^{0})-h(n^{0})\right), \partial^{\alpha}{\rm div}U^{\kappa}\right)\notag\\
&=\frac{\eta(1+\varepsilon)}{\tau}\sum\limits_{\beta\leq\alpha,|\beta|=1}\left(\partial^{\alpha-\beta}\left(h'(N^{\kappa}+n^{0})\partial^{\beta}N^{\kappa}+
\left(h'(N^{\kappa}+n^{0})-h'(n^{0})\right)\partial^{\beta}n^{0}\right), \partial^{\alpha}{\rm div}U^{\kappa}\right)\notag\\
&=\frac{\eta(1+\varepsilon)}{\tau}\left(h'(N^{\kappa}+n^{0})\partial^{\alpha}N^{\kappa}, \partial^{\alpha}{\rm div}U^{\kappa}\right)+\frac{\eta(1+\varepsilon)}{\tau}\left(\mathcal{H}_{U}^{(31)}, \partial^{\alpha}{\rm div}U^{\kappa}\right)\notag\\
&\quad+\frac{\eta(1+\varepsilon)}{\tau}\sum\limits_{\beta\leq\alpha,|\beta|=1}\left(\partial^{\alpha-\beta}\left(
\left(h'(N^{\kappa}+n^{0})-h'(n^{0})\right)\partial^{\beta}n^{0}\right), \partial^{\alpha}{\rm div}U^{\kappa}\right)\notag\\
&=:I^{(31)}+I^{(32)}+I^{(33)},
\end{align}
where
$$\mathcal{H}_{U}^{(31)}=\sum\limits_{\beta\leq\alpha,|\beta|=1}\partial^{\alpha-\beta}\left(h'(N^{\kappa}+n^{0})\partial^{\beta}N^{\kappa}\right)-h'(N^{\kappa}+n^{0})\partial^{\alpha}N^{\kappa}.$$
Applying the Moser-type inequalities in Lemma \ref{lem1.1}, we have
\begin{align*}
\left\|\mathcal{H}^{(31)}_{U}\right\|
&\lesssim\left\|Dh'(N^{\kappa}+n^{0})\right\|_{\infty}\left\|D^{l-1}\partial^{\beta}N^{\kappa}\right\|
+\left\|\partial^{\beta}N^{\kappa}\right\|_{\infty}\left\|D^{l}h'(N^{\kappa}+n^{0})\right\|\\
&\lesssim\left(1+\left\|N^{\kappa}\right\|_{l}\right)\left\|N^{\kappa}\right\|_{l}.
\end{align*}
Therefore, for the term $I^{(32)}$, one has
\begin{align}\label{i32}
I^{(32)}\lesssim\left\|\partial^{\alpha}{\rm div}U^{\kappa}\right\|\left(\left\|W^{\kappa}\right\|_{l}+\left\|W^{\kappa}\right\|_{l}^{2}\right).
\end{align}
For the term $I^{(33)}$, integrating by parts yields that
\begin{align}\label{i33}
I^{(33)}\lesssim\left\|U^{\kappa}\right\|_{l}\left\|N^{\kappa}\right\|_{l}\lesssim\left\|W^{\kappa}\right\|_{l}^{2}.
\end{align}
For the term $I^{(31)}$, using the equation \eqref{2.1}, we divide it into three parts:
\begin{align}\label{i31bds}
I^{(31)}&=\left(h'(N^{\kappa}\!+\!n^{0})\partial^{\alpha}N^{\kappa},\partial^{\alpha}\!\left(\frac{1}{N^{\kappa}+n^{0}}
\left((1+\varepsilon)\partial_{t}N^{\kappa}\!+\!U^{\kappa}\cdot
\nabla(N^{\kappa}\!+\!n^{0})\!+\!{\rm div}(N^{\kappa}u^{0})\right)\right)\!\right)\notag\\
&\quad \times\left(-\frac{\eta(1+\varepsilon)}{\tau}\right)\\
&=:I^{(311)}+I^{(312)}+I^{(313)}.\notag
\end{align}
Firstly, we deal with the term $I^{(311)}$. To this end, we rewrite $I^{(311)}$ as
\begin{align}\label{i311bds}
I^{(311)}&=-\frac{\eta(1+\varepsilon)^2}{\tau}\left(h'(N^{\kappa}+n^{0})\partial^{\alpha}N^{\kappa},\frac{\partial^{\alpha}\partial_{t}N^{\kappa}}{N^{\kappa}+n^{0}}\right)
-\frac{\eta(1+\varepsilon)^2}{\tau}\left(h'(N^{\kappa}+n^{0})\partial^{\alpha}N^{\kappa},\mathcal{H}^{(32)}_{U^{\kappa}}\right)\notag\\
&=-\frac{\eta(1+\varepsilon)^2}{2\tau}\frac{d}{dt}\!\!\int\frac{h'(N^{\kappa}+n^{0})}{N^{\kappa}+n^{0}}|\partial^{\alpha}N^{\kappa}|^{2}dx
+\frac{\eta(1+\varepsilon)^2}{2\tau}\!\!\int\!\partial_{t}\left(\frac{h'(N^{\kappa}+n^{0})}{N^{\kappa}+n^{0}}\right)|\partial^{\alpha}N^{\kappa}|^{2}dx\notag\\
&\quad-\frac{\eta(1+\varepsilon)^2}{\tau}\left(h'(N^{\kappa}+n^{0})\partial^{\alpha}N^{\kappa},\mathcal{H}^{(32)}_{U^{\kappa}}\right),
\end{align}
where
\begin{align*}
\mathcal{H}^{(32)}_{U^{\kappa}}=\partial^{\alpha}\left(\frac{\partial_{t}N^{\kappa}}{N^{\kappa}+n^{0}}\right)-\frac{1}{N^{\kappa}+n^{0}}\partial^{\alpha}\partial_{t}N^{\kappa}.
\end{align*}
Combining the Moser-type inequalities in Lemma \ref{lem1.1} and the equation \eqref{2.1}, we have
\begin{align*}
\left\|\mathcal{H}^{(32)}_{U^{\kappa}}\right\|&\lesssim\left\|D(\frac{1}{N^{\kappa}+n^{0}})\right\|_{\infty}\left\|D^{l-1}\partial_{t}N^{\kappa}\right\|
+\left\|\partial_{t}N^{\kappa}\right\|_{\infty}\left\|D^{l}(\frac{1}{N^{\kappa}+n^{0}})\right\|\\
&\lesssim\left(1+\|N^{\kappa}\|_{l}\right)\left(\|U^{\kappa}\|_{l}+\|N^{\kappa}\|_{l}+\|U^{\kappa}\|_{l}\|N^{\kappa}\|_{l}\right)\\
&\lesssim\|W^{\kappa}\|_{l}+\|W^{\kappa}\|^{2}_{l}+\|W^{\kappa}\|^{3}_{l},
\end{align*}
from which we can infer that the last term on the right-hand side of \eqref{i311bds} can be controlled by
$$C\left(\|W^{\kappa}\|^{2}_{l}+\|W^{\kappa}\|^{3}_{l}+\|W^{\kappa}\|^{4}_{l}\right).$$
Then, we deal with the second term on the right-hand side of \eqref{i311bds}. By using the equation \eqref{3.2'}, the smoothness of $h$ and the bounds on $N^{\kappa}+n^{0}$, we deduce
\begin{align*}
\left\|\partial_{t}\left(\frac{h'(N^{\kappa}+n^{0})}{N^{\kappa}+n^{0}}\right)\right\|_{\infty}&=\left\|\frac{h''(N^{\kappa}+n^{0})}{N^{\kappa}+
n^{0}}\partial_{t}(N^{\kappa}+n^{0})-\frac{h'(N^{\kappa}+n^{0})}{(N^{\kappa}+n^{0})^{2}}\partial_{t}(N^{\kappa}+n^{0})\right\|_{\infty}\\
&\lesssim\left\|{\rm div}\left((N^{\kappa}+n^{0})(U^{\kappa}+u^{0})\right)\right\|_{\infty}\\
&\lesssim 1+\|U^{\kappa}\|_{l}+\|N^{\kappa}\|_{l}+\|U^{\kappa}\|_{l}\|N^{\kappa}\|_{l}.
\end{align*}
Thus, we obtain
\begin{align}\label{i311}
I^{(311)}\leq-\frac{\eta(1+\varepsilon)^2}{\tau}\frac{1}{2}\frac{d}{dt}\int\frac{h'(N^{\kappa}+n^{0})}{N^{\kappa}+n^{0}}\left|\partial^{\alpha}N^{\kappa}\right|^{2}dx
+C\left(\|W^{\kappa}\|^{2}_{l}+\|W^{\kappa}\|^{3}_{l}+\|W^{\kappa}\|^{4}_{l}\right).
\end{align}
Next, we deal with the term $I^{(312)}$. Integrating by parts, using the regularity of $(n^0, u^0)$, Sobolev's imbedding theorem and the bounds on $N^{\kappa}+n^{0}$, we have
\begin{align*}
I^{(312)}&=-\frac{\eta(1+\varepsilon)}{\tau}\left(h'(N^{\kappa}+n^{0})\partial^{\alpha}N^{\kappa},\partial^{\alpha}\left(\frac{1}{N^{\kappa}+n^{0}}
U^{\kappa}\cdot\nabla\left(N^{\kappa}+n^{0}\right)\right)\right)\\
&=-\frac{\eta(1+\varepsilon)}{\tau}\left(h'(N^{\kappa}+n^{0})\partial^{\alpha}N^{\kappa},\frac{1}{N^{\kappa}+n^{0}}
U^{\kappa}\cdot\partial^{\alpha}\nabla\left(N^{\kappa}+n^{0}\right)\right)\\
&\quad-\frac{\eta(1+\varepsilon)}{\tau}\left(h'(N^{\kappa}+n^{0})\partial^{\alpha}N^{\kappa},\mathcal{H}^{(33)}_{U^{\kappa}}\right)\\
&=-\frac{\eta(1+\varepsilon)}{\tau}\left(h'(N^{\kappa}+n^{0})\partial^{\alpha}N^{\kappa},\frac{U^{\kappa}\cdot\partial^{\alpha}\nabla n^{0}}{N^{\kappa}+n^{0}}\right)\\
&\quad+\frac{1}{2}\frac{\eta(1+\varepsilon)}{\tau}\left(\partial^{\alpha}N^{\kappa},\partial^{\alpha}N^{\kappa}{\rm div}\left(\frac{h'(N^{\kappa}+n^{0})U^{\kappa}}{N^{\kappa}+n^{0}}\right)\right)\!-\!\frac{\eta(1+\varepsilon)}{\tau}
\left(h'(N^{\kappa}+n^{0})\partial^{\alpha}N^{\kappa},\mathcal{H}^{(33)}_{U^{\kappa}}\right)\\
&\lesssim\|N^{\kappa}\|_{l}\|U^{\kappa}\|_{l}+\|N^{\kappa}\|^2_{l}\|U^{\kappa}\|_{l}+\|N^{\kappa}\|_{l}\left\|\mathcal{H}^{(33)}_{U^{\kappa}}\right\|,
\end{align*}
where
$$\mathcal{H}^{(33)}_{U^{\kappa}}=\partial^{\alpha}\left(\frac{1}{N^{\kappa}+n^{0}}
U^{\kappa}\cdot\nabla\left(N^{\kappa}+n^{0}\right)\right)-\frac{1}{N^{\kappa}+n^{0}}U^{\kappa}\cdot\nabla\partial^{\alpha}\left(N^{\kappa}+n^{0}\right).$$
Furthermore, in view of Moser-type inequalities in Lemma \ref{lem1.1}, we have
\begin{align*}
\left\|\mathcal{H}^{(33)}_{U^{\kappa}}\right\|&\lesssim\left\|D\left(\frac{U^{\kappa}}{N^{\kappa}+n^{0}}\right)\right\|_{\infty}
\left\|D^{l-1}\nabla\left(N^{\kappa}+n^{0}\right)\right\|+\left\|\nabla\left(N^{\kappa}+n^{0}\right)\right\|_{\infty}
\left\|D^{l}\left(\frac{U^{\kappa}}{N^{\kappa}+n^{0}}\right)\right\|\\
&\lesssim\left(1+\|N^{\kappa}\|_{l}\right)^{2}\|U^{\kappa}\|_{l}.
\end{align*}
Therefore, one deduces that
\begin{align}\label{i312}
I^{(312)}\lesssim\left\|W^{\kappa}\right\|_{l}^{2}+\left\|W^{\kappa}\right\|_{l}^{3}+\left\|W^{\kappa}\right\|_{l}^{4}.
\end{align}

Now, we deal with the term $I^{(313)}$. Integrating by parts,  utilizing the regularity of $(n^0, u^0)$, Sobolev's imbedding theorem and the bounds on $N^{\kappa}$ and $n^{0}$, we have
\begin{align*}
I^{(313)}&=-\frac{\eta(1+\varepsilon)}{\tau}\left(h'(N^{\kappa}+n^{0})\partial^{\alpha}N^{\kappa},\partial^{\alpha}\left(\frac{1}{N^{\kappa}+n^{0}}
N^{\kappa}{\rm div}u^{0}\right)\right)\\
&\quad-\frac{\eta(1+\varepsilon)}{\tau}\left(h'(N^{\kappa}+n^{0})\partial^{\alpha}N^{\kappa},
\partial^{\alpha}\left(\frac{1}{N^{\kappa}+n^{0}}u^{0}\cdot\nabla N^{\kappa}\right)\right)\\
&\leq C\|N^{\kappa}\|^2_{l}
-\frac{\eta(1+\varepsilon)}{\tau}\left(h'(N^{\kappa}+n^{0})\partial^{\alpha}N^{\kappa},\frac{u^{0}\cdot\nabla\partial^{\alpha}N^{\kappa}}{N^{\kappa}+n^{0}}\right)\\
&\quad-\frac{\eta(1+\varepsilon)}{\tau}\left(h'(N^{\kappa}+n^{0})\partial^{\alpha}N^{\kappa},\mathcal{H}^{(34)}_{U^{\kappa}}\right)\\
&\leq C\|N^{\kappa}\|^2_{l}
+\frac{1}{2}\frac{\eta(1+\varepsilon)}{\tau}\left(\partial^{\alpha}N^{\kappa},\partial^{\alpha}N^{\kappa}{\rm div}\left(\frac{h'(N^{\kappa}+n^{0})}{N^{\kappa}+n^{0}}u^{0}\right)\right)\\
&\quad-\frac{\eta(1+\varepsilon)}{\tau}\left(h'(N^{\kappa}+n^{0})\partial^{\alpha}N^{\kappa},\mathcal{H}^{(34)}_{U^{\kappa}}\right)\\
&\lesssim \|N^{\kappa}\|^2_{l}+\|N^{\kappa}\|_{l}\left\|\mathcal{H}^{(34)}_{U^{\kappa}}\right\|,
\end{align*}
where
$$\mathcal{H}^{(34)}_{U^{\kappa}}=\partial^{\alpha}\left(\frac{u^{0}\cdot\nabla N^{\kappa}}{N^{\kappa}+n^{0}}
\right)-\frac{u^{0}\cdot\nabla \partial^{\alpha}N^{\kappa}}{N^{\kappa}+n^{0}}.$$
Applying Moser-type inequalities in Lemma \ref{lem1.1} again, we infer that
\begin{align*}
\left\|\mathcal{H}^{(34)}_{U^{\kappa}}\right\|&\lesssim\left\|D\left(\frac{u^{0}}{N^{\kappa}+n^{0}}\right)\right\|_{\infty}
\left\|D^{l-1}\nabla N^{\kappa}\right\|+\left\|\nabla N^{\kappa}\right\|_{\infty}
\left\|D^{l}\left(\frac{u^{0}}{N^{\kappa}+n^{0}}\right)\right\|\\
&\lesssim\left(1+\|N^{\kappa}\|_{l}\right)\|N^{\kappa}\|_{l}.
\end{align*}
Therefore, one has
\begin{align}\label{i313}
I^{(313)}\lesssim\left\|W^{\kappa}\right\|_{l}^{2}+\left\|W^{\kappa}\right\|_{l}^{3}.
\end{align}
Combining \eqref{i311}, \eqref{i312} and \eqref{i313}, we obtain
\begin{align}\label{i31}
I^{(31)}\leq -\frac{\eta(1+\varepsilon)^2}{\tau}\frac{1}{2}\frac{d}{dt}\int\frac{h'(N^{\kappa}+n^{0})}{N^{\kappa}+n^{0}}\left|\partial^{\alpha}N^{\kappa}\right|^{2}dx
+C\left(\left\|W^{\kappa}\right\|_{l}^{2}+\left\|W^{\kappa}\right\|_{l}^{3}+\left\|W^{\kappa}\right\|_{l}^{4}\right).
\end{align}
Substituting \eqref{i32}, \eqref{i33}, \eqref{i31} into \eqref{i3bds}, we have
\begin{align}\label{i3}
I^{(3)}&\leq -\frac{\eta(1+\varepsilon)^2}{\tau}\frac{1}{2}\frac{d}{dt}\int\frac{h'(N^{\kappa}+n^{0})}{N^{\kappa}+n^{0}}\left|\partial^{\alpha}N^{\kappa}\right|^{2}dx\notag\\
&\quad+C\left\|\partial^{\alpha}{\rm div}U^{\kappa}\right\|\left(\left\|W^{\kappa}\right\|_{l}+\left\|W^{\kappa}\right\|_{l}^{2}\right)
+C\left(\left\|W^{\kappa}\right\|_{l}^{2}+\left\|W^{\kappa}\right\|_{l}^{3}+\left\|W^{\kappa}\right\|_{l}^{4}\right).
\end{align}
By the regularity of $\left(n^{0},u^{0}\right)$ and Sobolev's imbedding theorem, we get
\begin{align}\label{i4}
I^{(4)}\lesssim\left\| U^{\kappa}\right\|_{l+1}\left\| U^{\kappa}\right\|_{l}\left(1+\left\| U^{\kappa}\right\|_{l}\right)
\lesssim\left\| U^{\kappa}\right\|_{l+1}\left(\left\|W^{\kappa}\right\|_{l}+\left\|W^{\kappa}\right\|_{l}^{2}\right).
\end{align}
Similarly, we have
\begin{align}\label{i5}
I^{(5)}\lesssim\left\| U^{\kappa}\right\|^2_{l},
\end{align}
\begin{align}\label{i9}
I^{(9)}\lesssim\kappa\left\|\kappa\tilde{j}^{\kappa}\right\|_l\left\|B^\kappa\right\|_l\left\| U^{\kappa}\right\|_{l}.
\end{align}
Integrating by parts, the regularity of $(n^0, u^0)$, Sobolev's imbedding theorem and the bounds on $N^{\kappa}$ and $n^{0}$ imply that
\begin{align}\label{i6}
I^{(6)}\lesssim\left\| U^{\kappa}\right\|_{l+1}\left\|\kappa\tilde{j}^{\kappa}\right\|_l,
\end{align}
\begin{align}\label{i78}
I^{(7)},I^{(8)}\lesssim\left\| U^{\kappa}\right\|_{l+1}\left\| N^{\kappa}\right\|_{l}.
\end{align}

Inserting \eqref{i1}, \eqref{i2}, \eqref{i3}, \eqref{i4}--\eqref{i78} into \eqref{hoe}, we conclude that
\begin{align}\label{hoeresult}
\begin{split}
\frac{1}{2}&\frac{d}{dt}\left\|\partial^{\alpha}U^{\kappa}\right\|^{2}+\mu\int\frac{1}{N^{\kappa}+n^{0}}\left|\partial^{\alpha}\nabla U^{\kappa}\right|^{2}dx
+(\mu+\lambda)\int\frac{1}{N^{\kappa}+n^{0}}\left|\partial^{\alpha}{\rm div} U^{\kappa}\right|^{2}dx\\
&\quad+\frac{\eta(1+\varepsilon)^2}{\tau}\frac{1}{2}\frac{d}{dt}\int\frac{h'(N^{\kappa}+n^{0})}{N^{\kappa}+n^{0}}\left|\partial^{\alpha}N^{\kappa}\right|^{2}dx\\
&\leq \sigma_{\alpha}\|U^{\kappa}\|_{l+1}^{2}+C({\sigma_{\alpha}})\left(\left\|W^{\kappa}\right\|_{l}^{2}+\left\|W^{\kappa}\right\|_{l}^{3}+\left\|W^{\kappa}\right\|_{l}^{4}+
\left\|W^{\kappa}\right\|_{l}^{5}+\left\|W^{\kappa}\right\|_{l}^{6}\right),
\end{split}
\end{align}
for some small constant $\sigma_{\alpha}$, which will be determined later.

Applying a same argument on \eqref{2.3}, one has
\begin{align}\label{jresult}
\begin{split}
\frac{1}{2}&\frac{d}{dt}\left\|\kappa\partial^{\alpha}\tilde{j}^{\kappa}\right\|^{2}+\mu\int\frac{1}{N^{\kappa}+n^{0}}
\left|\kappa\partial^{\alpha}\nabla\tilde{j}^{\kappa}\right|^{2}dx
+(\mu+\lambda)\int\frac{1}{N^{\kappa}+n^{0}}\left|\kappa\partial^{\alpha}{\rm div}\tilde{j}^{\kappa}\right|^{2}dx\\
&\leq \sigma_{\alpha}\|\kappa\tilde{j}^{\kappa}\|_{l+1}^{2}+C({\sigma_{\alpha}})\left(\left\|W^{\kappa}\right\|_{l}^{2}+
\left\|W^{\kappa}\right\|_{l}^{3}+\left\|W^{\kappa}\right\|_{l}^{4}+
\left\|W^{\kappa}\right\|_{l}^{5}+\left\|W^{\kappa}\right\|_{l}^{6}\right).
\end{split}
\end{align}
Employing the operator $\partial^{\alpha}$ to \eqref{2.4} and \eqref{2.5},  and then taking the inner product of the resulting equations, $\partial^{\alpha}E^{\kappa}$ and $\partial^{\alpha}G^{\kappa}$ respectively, we have
\begin{align}\label{egresult}
\frac{1}{2}\frac{d}{dt}\left(\left\|\partial^{\alpha}E^{\kappa}\right\|^{2}+\left\|\partial^{\alpha}G^{\kappa}\right\|^{2}\right)
\lesssim\left\|W^{\kappa}\right\|_{l}^{2}+\left\|W^{\kappa}\right\|_{l}^{3}.
\end{align}
Applying the bounds on $N^{\kappa}+n^{0}$ stated in Proposition \ref{solu n0u0} and Proposition \ref{solu nk}, combining \eqref{hoeresult}, \eqref{jresult}, \eqref{egresult} and Lemma \ref{lem1}, summing $\alpha$ over $1\leq\alpha\leq l$, and choosing $\sigma_{\alpha}$ small enough, we obtain the desired result of Lemma \ref{lem2}.
\end{proof}

\subsection{Proof of Theorem \ref{main result}}\label{sec4}
Now, we are ready to prove Theorem \ref{main result}.
\begin{proof}
As in \cite{jf}, set $\Gamma^{\kappa}(t)=\left\|W^{\kappa}(t)\right\|^{2}_{l}$. By Lemma \ref{lem2}, for any $t\in(0,T)$ and sufficiently small $\kappa$, we have
\begin{align}\label{5.1}
\Gamma^{\kappa}(t)\leq \Gamma^{\kappa}(0)+C\int_{0}^{t}\left(\Gamma^{\kappa}\left(1+\left(\Gamma^{\kappa}\right)^{\frac{1}{2}}+
\Gamma^{\kappa}+\left(\Gamma^{\kappa}\right)^{\frac{3}{2}}+\left(\Gamma^{\kappa}\right)^{2}\right)\right)(\tau)d\tau.
\end{align}
Choose $\kappa_{1}>0$ and $T_{1}\in(0,1)$ small enough such that
\begin{align}\label{5.2}
1+\left(3C\right)^{\frac{1}{2}}\kappa_{1}+3C\kappa_{1}^2+\left(3C\right)^{\frac{3}{2}}\kappa_{1}^3+\left(3C\right)^{2}\kappa_{1}^4\leq2,\;\;1+2CT_1e^{2CT_1}\leq\frac{3}{2}.
\end{align}
Suppose that $\Gamma^{\kappa}(t)\leq3C\kappa^2$. Noted that by using the assumption $\Gamma^{\kappa}(0)\leq C_{0}\kappa^{2}$ in Theorem\ref{main result}, \eqref{5.1} and \eqref{5.2}, for any $\kappa\in(0,\kappa_{1})$ and $t\in(0, T_{1}]$, we have
\begin{align}\label{5.3}
\begin{split}
\Gamma^{\kappa}(t)&\leq C\kappa^2+
C\int_{0}^{t}\Gamma^{\kappa}(\tau)\left(1+\left(3C\right)^{\frac{1}{2}}\kappa+3C\kappa^2+\left(3C\right)^{\frac{3}{2}}\kappa^3+\left(3C\right)^{2}\kappa^4\right)d\tau\\
&\leq C\kappa^2+2C\int_{0}^{t}\Gamma^{\kappa}(\tau)d\tau,
\end{split}
\end{align}
which combined with Gronwall's inequality and \eqref{5.2} implies that
\begin{align}\label{5.4}
\Gamma^{\kappa}(t)\leq C\kappa^2\left(1+2Cte^{2Ct}\right)\leq C\kappa^2\left(1+2CT_1e^{2CT_1}\right)\leq\frac{3}{2}C\kappa^2.
\end{align}
By the bootstrap principle, we conclude that $\Gamma^{\kappa}(t)\leq3C\kappa^2$ for any $\kappa\in(0,\kappa_{1})$ and $t\in(0, T_{1}]$. Hence, by standard continuous induction method, for any $T_{0}< T^{*}$, there exists a positive constant $\kappa_{0}(T_{0})$, such that for any $\kappa\in(0,\kappa_{0})$, $T_{\kappa}\geq T_{0}$ and $\Gamma^{\kappa}(t)\leq \tilde{C}\kappa^{2}$ on $[0,T_{0}]$, for some constant $\tilde{C}$ depending only on $T_{0}$ and  the initial data. Therefore, we complete the proof of Theorem \ref{main result}.
\end{proof}

\section*{Acknowledgments}
The research was supported by the National Natural Science Foundation of China (No. 11901066), the
Natural Science Foundation of Chongqing (No. cstc2019jcyj-msxmX0167) and projects No. 2019CDXYST0015,  No. 2020CDJQY-A040 supported by the Fundamental Research Funds for the Central Universities.


\begin{thebibliography}{aa}
\footnotesize


\bibitem{bdd}
C. Besse, P. Degond, F. Deluzet, A model of hierarchy for ionospheric plasma modeling, Math. Models Methods Appl., 14(3) (2004) 393-415.

\bibitem{cjw}
G.Q. Chen, J.W. Jerome, D.H. Wang, Compressible Euler-Maxwell equations, Proceedings of the Fifth International Workshop on Mathematical Aspects of Fluid and Plasma Dynamics (Maui, HI, 1998), Transport Theory and Statistical Physics, 29(3-5) (2000) 311-331.

\bibitem{ghr}
D. G\'{e}rard-Varet, D. Han-Kwan, F. Rousset, Quasineutral limit of the Euler-Poisson system for ions in a domain with boundaries, Indiana University Mathematics Journal, 62(2) (2013) 359-402.

\bibitem{ghr1}
D. G\'{e}rard-Varet, D. Han-Kwan, F. Rousset, Quasineutral limit of the Euler-Poisson system for ions in a domain with boundaries II, J. \'{E}c. Polytech., 1 (2014) 343-386.

\bibitem{jf}
S. Jiang, F.C. Li, Rigorous derivation of compressible magnetohydrodynamic equations from the electromagnetic fluid system, Nonlinearity, 25(6) (2012) 1735-1752.

\bibitem{jjll}
S. Jiang, Q. Ju, H. Li, Y. Li, Quasi-neutral limit of the full bipolar Euler-Poisson system, Science China Mathematics, 53(12) (2010) 3099-3114.

\bibitem{jl}
Q.C. Ju, Y. Li, Quasineutral limit of the two-fluid Euler-Poisson system in a bounded domain of $\mathbb{R}^{3}$, Journal of Mathematical Analysis and Applications, 469(1) (2019) 169-187.

\bibitem{jllj}
Q. Ju, H. Li, Y. Li, S. Jiang, Quasi-neutral limit of the two-fluid Euler-Poisson system, Communications on Pure Applied Analysis, 9(6) (2010) 1577-1590.

\bibitem{k1}
T. Kato, Nonstationary flows of viscous and ideal fluids in $\mathbb{R}^3$, Journal of Functional Analysis, 9(3) (1972) 296-305.

\bibitem{k}
S. Kawashima, Systems of a hyperbolic-parabolic composite type, with applications to the equations of magnetohydrodynamics, Kyoto University, 1984.

\bibitem{km}
S. Klainerman, A. Majda, Compressible and incompressible fluids, Communications on Pure Applied Mathematics, 35(5) (1982) 629-651.

\bibitem{km1}
S. Klainerman, A. Majda, Singular limits of quasilinear hyperbolic systems with large parameters and the incompressible limit of compressible fluids, Communications on Pure and Applied Mathematics, 34(4) (1981) 481-524.

\bibitem{l}
Y.P. Li, The asymptotic behavior and the quasineutral limit for the bipolar Euler-Poisson system with boundary effects and a vacuum, Chinese Annals of Mathematics, Series B, 34(4) (2013) 529-540.

\bibitem{lpw}
Y. Li, Y.J. Peng, Y.G. Wang, From two-fluid Euler-Poisson equations to one-fluid Euler equations, Asymptotic Analysis, 85(3-4) (2013) 125-148.

\bibitem{lpx}
Y.C. Li, Y.J. Peng, S. Xi, The combined non-relativistic and quasi-neutral limit of two-fluid Euler-Maxwell equations, Z. Angew. Math. Phys., 66(6) (2015) 3249-3265.

\bibitem{m}
F.J. McGrath, Nonstationary plane flow of viscous and ideal fluids, Archive for Rational Mechanics and Analysis, 27(5) (1968) 329-348.

\bibitem{pw1}
Y.J. Peng, S. Wang, Asymptotic expansions in two-fluid compressible Euler-Maxwell equations with small parameters, Discrete and continuous dynamical systems, 23(1-2) (2009) 415-433.

\bibitem{pw3}
Y.J. Peng, S. Wang, Convergence of compressible Euler-Maxwell equations to compressible Euler-Poisson equations, Chinese Annals of Mathematics, Series B, 28(5) (2007) 583-602.

\bibitem{pw2}
Y.J. Peng, S. Wang, Convergence of compressible Euler-Maxwell equations to incompressible Euler equations, Comm. Part. Diff. Eqs., 33(3) (2008) 349-367.

\bibitem{pw}
Y.J. Peng, S. Wang, Rigorous derivation of imcompressible e-MHD equations from compressible Euler-Maxwell equations, SIAM J. Math. Anal., 40(2) (2008) 540-565.

\bibitem{pw4}
Y.J. Peng, Y.G. Wang, Boundary layers and quasi-neutral limit in steady state Euler-Poisson equations for potential flows, Nonlinearity, 17(3) (2004) 835-849.

\bibitem{pwg}
Y.J. Peng, S. Wang, Q.L. Gu, Relaxation limit and global existence of smooth solutions of compressible Euler-Maxwell equations, SIAM J. Math. Anal., 43(2) (2011) 944-970.

\bibitem{v}
M.H. Vignal, A boundary layer problem for an asymptotic preserving scheme in the quasi-neutral limit for the Euler-Poisson system, SIAM J. Math. Anal., 70(6) (2010) 1761-1787.

\bibitem{vh}
A.I. Vol'pert, S.I. Hujeav, On the cauchy problem for composite systems of nonlinear differential equations, Mat. Sbornik, 16(4) (1972) 517-544.

\bibitem{x}
L.J. Xiong, Incompressible limit of isentropic Navier-Stokes equations with Navier-slip boundary, Kinet. Relat. Models, 11(3) (2018) 469-490.

\bibitem{yw1}

J.W. Yang, S. Wang, Convergence of compressible Navier-Stokes-Maxwell equations to incompressible Navier-Stokes equations, Science China Mathematics, 57(10) (2014) 2153-2162.

\bibitem{yw2}
J. Yang, S. Wang, Convergence of the Euler-Maxwell two-fluid system to compressible Euler equations, Journal of Mathematical Analysis and Applications, 417(2) (2014) 889-903.

\bibitem{yw}
J.W. Yang, S. Wang, Non-relativistic limit of two-fluid Euler-Maxwell equations arising from plasma physics, Z. Angew. Math. Mech., 89(12) (2009) 981-994.

\bibitem{z}
S. Zheng, Nonlinear parabolic equations and hyperbolic-parabolic coupled systems, CRC Press, 1995.







\end{thebibliography}
\end{document}